\documentclass[11pt,leqno]{article}
\usepackage{hyperref}
\usepackage{soul}
\usepackage{pdfsync}
\usepackage[doc]{optional}
\usepackage{xcolor}
\definecolor{labelkey}{rgb}{0,0.08,0.45}
\definecolor{rekey}{rgb}{0,0.6,0.0}
\definecolor{Brown}{rgb}{0.45,0.0,0.05}
\usepackage{exscale,relsize}
\usepackage{amsmath}
\usepackage{amsfonts}
\usepackage{amssymb}
\usepackage{calc}
\usepackage{theorem}
\usepackage{pifont}      
\usepackage{graphicx}
\DeclareMathOperator{\weakstarly}{\rightharpoondown_{\mathrm{w*}}}

\newcommand{\scal}[2]{\langle{{#1},{#2}}\rangle}

\newcommand{\RR}{\ensuremath{\mathbb R}}

\newcommand{\RX}{\ensuremath{\,\left]-\infty,+\infty\right]}}
\newcommand{\RXX}{\ensuremath{\,\left[-\infty,+\infty\right]}}

\newcommand{\menge}[2]{\big\{{#1} \mid {#2}\big\}}

\newcommand{\To}{\ensuremath{\rightrightarrows}}

\newcommand{\dom}{\ensuremath{\operatorname{dom}}}

\newcommand{\gra}{\ensuremath{\operatorname{gra}}}

\newcommand{\inte}{\ensuremath{\operatorname{int}}}

\newcommand{\ran}{\ensuremath{\operatorname{ran}}}

\newcommand{\conv}{\ensuremath{\operatorname{conv}}}

\renewcommand{\phi}{\ensuremath{\varphi}}

\newtheorem{theorem}{Theorem}[section]
\newtheorem{lemma}[theorem]{Lemma}
\newtheorem{fact}[theorem]{Fact}
\newtheorem{corollary}[theorem]{Corollary}
\newtheorem{proposition}[theorem]{Proposition}
\newtheorem{definition}[theorem]{Definition}

\theoremstyle{plain}{\theorembodyfont{\rmfamily}
}
\theoremstyle{plain}{\theorembodyfont{\rmfamily}
}
\theoremstyle{plain}{\theorembodyfont{\rmfamily}
}
\theoremstyle{plain}{\theorembodyfont{\rmfamily}
}
\theoremstyle{plain}{\theorembodyfont{\rmfamily}
\newtheorem{remark}[theorem]{Remark}}
\newtheorem{problem}[theorem]{Problem}
\theoremstyle{plain}{\theorembodyfont{\rmfamily}
}

\begin{document}


\title{\sffamily{Characterizations of ultramaximally monotone operators}}

\author{
 Liangjin\ Yao\thanks{CARMA,
University of Newcastle,
 Newcastle, New South Wales 2308, Australia.
E-mail:  \texttt{liangjinyao@gmail.com}.}}

\date{January 28,  2014}
\maketitle

\begin{abstract} \noindent

In this paper, we study properties of
 ultramaximally monotone operators.
 We characterize the interior and
 the closure of the range of an ultramaximally monotone
operator. We establish
the Brezis--Haraux condition in
the setting of a general Banach space. Moreover, we show that every  ultramaximally monotone operator is
of type (NA).
We also  provide  some sufficient conditions for
 a Banach space to be reflexive
 by a linear continuous and ultramaximally monotone operator.

\end{abstract}



\section{Introduction}

Throughout this paper, we assume that
$X$ is a real Banach space with norm $\|\cdot\|$,
that $X^*$ is the continuous dual of $X$, and
that $X$ and $X^*$ are paired by $\scal{\cdot}{\cdot}$.
Let $A\colon X\To X^*$
be a \emph{set-valued operator}  (also known as a relation,
 point-to-set mapping or multifunction)
from $X$ to $X^*$, i.e., for every $x\in X$, $Ax\subseteq X^*$,
and let
$\gra A := \menge{(x,x^*)\in X\times X^*}{x^*\in Ax}$ be
the \emph{graph} of $A$.
Recall that $A$ is  \emph{monotone} if
\begin{equation}
\scal{x-y}{x^*-y^*}\geq 0,\quad \forall (x,x^*)\in \gra A\;\forall (y,y^*)\in\gra A,
\end{equation}
and \emph{maximally monotone} if $A$ is monotone
 and $A$ has no proper monotone extension
(in the sense of graph inclusion).
Let $A:X\rightrightarrows X^*$ be monotone and $(x,x^*)\in X\times X^*$.
 We say $(x,x^*)$ is \emph{monotonically related to}
$\gra A$ if
\begin{align*}
\langle x-y,x^*-y^*\rangle\geq0,\quad \forall (y,y^*)\in\gra A.\end{align*}
Let $A:X\rightrightarrows X^*$ be monotone.
 We say $A$ is \emph{ultramaximally monotone}
 if $A$ is maximally monotone with respect
  to $X^{**}\times X^*$ \cite{BS1}.

 Our paper is motivated by Bauschke
 and Simons' paper \cite{BS1},
 where the authors study the
  properties of an unbounded
   linear ultramaximally monotone operator.
   In this paper,
   we continue to study
   the properties of a general ultramaximally monotone operator.

Monotone operators have proven to be important
objects in modern Optimization and Analysis; see, e.g.,
\cite{Bor1,Bor3, BY3, ABVY, Reval, YaoPhD} and the books
\cite{BC2011,BorVan,BurIus,ButIus,ph,Si,Si2,RockWets,Zalinescu,Zeidler2A,Zeidler2B}
and the references therein. We adopt standard notation used in these
books: $\dom A:= \menge{x\in X}{Ax\neq\varnothing}$ is the
\emph{domain} of $A$. Given a subset $C$ of $X$, $\inte C$ is the
\emph{interior} of $C$ and $\overline{C}$ is the
norm \emph{closure} of $C$. We set $C^{\bot}:= \{x^*\in X^*
\mid(\forall c\in C)\, \langle x^*, c\rangle=0\}$ and $S^{\bot}:=
\{x^{**}\in X^{**} \mid(\forall s\in S)\, \langle x^{**},
s\rangle=0\}$ for a set  $S\subseteq X^*$.
Let $A:X\rightrightarrows X^*$.
The \emph{adjoint} of an operator  $A$, written $A^*$, is defined by
\begin{equation*}
\gra A^* :=
\big\{(x^{**},x^*)\in X^{**}\times X^*\mid(x^*,-x^{**})\in(\gra A)^{\bot}\big\}.
\end{equation*} We say $A$ is a \emph{linear relation} if $\gra A$ is a
linear subspace.
Let  $A$ be a  linear relation.
  We say that $A$ is
\emph{skew} if $\gra A \subseteq \gra (-A^*)$;
equivalently, if $\langle x,x^*\rangle=0,\; \forall (x,x^*)\in\gra A$.
Furthermore,
$A$ is \emph{symmetric} if $\gra A
\subseteq\gra A^*$; equivalently, if $\scal{x}{y^*}=\scal{y}{x^*}$,
$\forall (x,x^*),(y,y^*)\in\gra A$.

The \emph{indicator function} of $C$, written as $\iota_C$, is defined
at $x\in X$ by
\begin{align}
\iota_C (x):=\begin{cases}0,\,&\text{if $x\in C$;}\\
\infty,\,&\text{otherwise}.\end{cases}\end{align}
If $D\subseteq X$, we set $C-D=\{x-y\mid x\in C, y\in D\}$.
  For every $x\in X$, the normal cone operator of $C$ at $x$
is defined by $N_C(x):= \menge{x^*\in
X^*}{\sup_{c\in C}\scal{c-x}{x^*}\leq 0}$, if $x\in C$; and $N_C(x)=\varnothing$,
if $x\notin C$.
For $x,y\in X$, we set $\left[x,y\right]:=\{tx+(1-t)y\mid 0\leq t\leq 1\}$.
 Given $f\colon X\to \RX$, we set
$\dom f:= f^{-1}(\RR)$ and
$f^*\colon X^*\to\RXX\colon x^*\mapsto
\sup_{x\in X}\big(\scal{x}{x^*}-f(x)\big)$ is
the \emph{Fenchel conjugate} of $f$.
We say $f$ is \emph{proper} if $\dom f\neq\varnothing$.
 We say $f$ is \emph{supercoercive} if $\lim_{\|x\|\rightarrow
\infty} \frac{f(x)}{\|x\|}=+\infty$.
Let $f$ be proper.
Then
   $\partial f\colon X\To X^*\colon
   x\mapsto \menge{x^*\in X^*}{(\forall y\in
X)\; \scal{y-x}{x^*} + f(x)\leq f(y)}$ is the \emph{subdifferential
operator} of $f$.
 We  denote  by $J_X$ the duality map from $X$ to $X^*$, i.e.,
the subdifferential of the function $\tfrac{1}{2}\|\cdot\|^2$.
 For convenience, we denote by $J:=J_X$.

 Let $A:X\rightrightarrows X^*$ and   $F\colon X\times X^*\rightarrow\RX$. We say
$F$ is a \emph{representer} for $\gra A$ if \begin{equation}
\gra A=\big\{(x,x^*)\in X\times X^*\mid F(x,x ^*)=\langle x, x^*\rangle\big\}.
\end{equation}
Given two real Banach spaces $X,Y$ and $F_1,
F_2\colon X\times Y\rightarrow\RX$, the \emph{partial
inf-convolution} $F_1\Box_2 F_2$ is the function defined on $X\times
Y$ by
\begin{equation*}F_1\Box_2 F_2\colon
(x,y)\mapsto \inf_{v\in Y} \big\{F_1(x,y-v)+F_2(x,v)\big\}.
\end{equation*} We set  $P_X: X\times Y\rightarrow
X\colon (x,y)\mapsto x$,
 and
 $P_Y: X\times Y\rightarrow Y\colon (x,y)\mapsto y$.
We denote by $\longrightarrow$ and
$\weakstarly$ the norm convergence and weak$^*$ convergence of
nets,  respectively.

We now recall two fundamental properties of maximally
monotone operators.

 \begin{definition}\label{trio}
 Let $A:X\To X^*$ be maximally monotone.
 \begin{enumerate}
 \item We say $A$ is
\emph{of dense type or type (D)} (1976, \cite{Gossez76};
 see also \cite{ph2} and \cite[Theorem~9.5]{Si5}) if for every
$(x^{**},x^*)\in X^{**}\times X^*$ with
\begin{align*}
\inf_{(a,a^*)\in\gra A}\langle a-x^{**}, a^*-x^*\rangle\geq 0,
\end{align*}
there exist a  bounded net
$(a_{\alpha}, a^*_{\alpha})_{\alpha\in\Gamma}$ in $\gra A$
such that
$(a_{\alpha}, a^*_{\alpha})_{\alpha\in\Gamma}$
weak*$\times$strong converges to
$(x^{**},x^*)$.
\item We say $A$ is
\emph{of type negative infimum (NI)} (1996, \cite{SiNI}) if
\begin{align*}
\inf_{(a,a^*)\in\gra A}\langle x^{**}-a,x^*-a^*\rangle\leq0,
\quad \forall(x^{**},x^*)\in X^{**}\times X^*.
\end{align*}

\end{enumerate}
\end{definition}
These two properties coincide  by Simons, Marques Alves and Svaite (see
\cite[Lemma~15]{SiNI} or \cite[Theorem~36.3(a)]{Si2}, and
\cite[Theorem~4.4]{MarSva}).
By the definition of ultrmaximally monotone operators,
 every ultramaximally monotone  operator is of type (D) (i.e., type (NI)).
   For convenience, we write type (NI) to represent type (D) as well.
Note that not every operator of type (NI) is ultramaximally monotone.
 For instance, suppose that  $X$ is a nonreflexive space. Let $A:X\rightrightarrows
X^*$ be defined by $\gra A:=X\times \{0\}$.
  Then $\gra A=\partial\iota_X$ and $A$ is of type (NI) by
Fact~\ref{GozztyD:1}. But $\gra A\varsubsetneq X^{**}\times \{0\}$.
 Hence $A$ is not ultramaximally monotone.

However,  there always exists an ultramximally monotone operator in $X$,
 for instance, the operator with the graph:
$\{0\}\times X^*$.

The remainder of this paper is organized as follows. In
Section~\ref{s:aux}, we collect auxiliary results for future
reference and for the reader's convenience. Our main
results (Theorem~\ref{PropLeP:1},  Theorem~\ref{domFdm:1},
Theorem~\ref{RectaE:1} and Theorem~\ref{TyNA:2}) are presented in  Section~\ref{s:main}.
We also pose various interesting open problems at the end of this paper.

\section{Auxiliary Results}
\label{s:aux}

In this section, we introduce some basic facts.
We start with the Attouch-Brezis' Fenchel duality theorem.
\begin{fact}[Attouch-Brezis]{\rm (See \cite[Theorem~1.1]{AtBrezis}
 or \cite[Remark~15.2]{Si2}.)}\label{AttBre:1}
Let $f,g: X\rightarrow\RX$ be proper lower semicontinuous  and convex.
Assume that $
\bigcup_{\lambda>0} \lambda\left[\dom f-\dom g\right]$
is a closed subspace of $X$.
Then $\partial (f+g)=\partial f+\partial g$ and
\begin{equation*}
(f+g)^*(z^*) =\min_{y^*\in X^*} \left[f^*(y^*)+g^*(z^*-y^*)\right],\quad \forall z^*\in X^*.
\end{equation*}
\end{fact}

\begin{fact}[Simons and Z\u{a}linescu]
\emph{(See \cite[Theorem~4.2]{SiZ} or
\cite[Theorem~16.4(a)]{Si2}.)}\label{F4} Let $X, Y$ be real Banach
spaces and $F_1, F_2\colon X\times Y \to \RX$ be proper lower
semicontinuous and convex. Assume that for every
$(x,y)\in X\times Y$,
\begin{equation*}(F_1\Box_2 F_2)(x,y)>-\infty
\end{equation*}
and that  $\bigcup_{\lambda>0} \lambda\left[P_X\dom F_1-P_X\dom F_2\right]$
is a closed subspace of $X$. Then for every $(x^*,y^*)\in X^*\times Y^*$,
\begin{equation*}
(F_1\Box_2 F_2)^*(x^*,y^*)=\min_{u^*\in X^*}
\left\{F_1^*(x^*-u^*,y^*)+F_2^*(u^*,y^*)\right\}.
\end{equation*}\end{fact}

The Fitzpatrick function below is a key tool in Monotone
Operator  Theory, which  has been applied comprehensively.
\begin{fact}[Fitzpatrick]
\emph{(See {\cite[Corollary~3.9]{Fitz88}}.)}
\label{f:Fitz}
Let $A\colon X\To X^*$ be maximally monotone,  and set
\begin{equation}
F_A\colon X\times X^*\to\RX\colon
(x,x^*)\mapsto \sup_{(a,a^*)\in\gra A}
\big(\scal{x}{a^*}+\scal{a}{x^*}-\scal{a}{a^*}\big),
\end{equation}
 the \emph{Fitzpatrick function} associated with $A$.
Then for every $(x,x^*)\in X\times X^*$, the inequality
$\scal{x}{x^*}\leq F_A(x,x^*)$ is true, and  equality holds if and
only if $(x,x^*)\in\gra A$.
\end{fact}

\begin{fact}[Fitzpatrick]\emph{(See
 \cite[Proposition~4.2 and Theorem~4.3]{Fitz88}.)}\label{fa:Fitza2}
Let $A:X\rightrightarrows X^*$ be a
 monotone operator with $\dom A\neq\varnothing$. Then
$F^{*}_A=\langle\cdot,\cdot\rangle$ on
 $\gra A^{-1}$, $F^{*}_A(x^*,x)\geq F_A(x,x^*),\,\forall(x,x^*)\in X\times X^*$,
 and
\begin{align*}
\big\{x\in X\mid \exists x^*\in X^*\,\text{such that}\,\,
 F^*_A(x^*,x)<+\infty\big\}\subseteq\overline{\conv(\dom A)}.
\end{align*}
\end{fact}

The following three results are
 the fundamental characterizations of the domain of a maximally monotone operator.
\begin{fact}[Simons]
\emph{(See \cite[Theorem~27.1]{Si2}.)}
\label{interodom}
Let $A:X\To X^*$ be a maximally monotone operator. Then
\begin{align*}
\inte{\dom A}=\inte\left[\conv \dom A\right]=\inte \left[P_X\dom F_A\big)\right].
\end{align*}
\end{fact}

\begin{fact}
\emph{(See \cite[Theorem~3.6]{BY2} or \cite{BY3}.)}
\label{CoHull}
Let $A:X\To X^*$ be a maximally monotone operator. Then
\begin{align*}
\overline{\conv\left[\dom A\right]}=\overline{P_X\left[\dom F_A\right]}.
\end{align*}
\end{fact}

Let $A:X\rightrightarrows X^*$ be maximally monotone. We say $A$ is
\emph{of type (FPV)} if  for every open convex set $U\subseteq X$ such that
$U\cap \dom A\neq\varnothing$, the implication
\begin{equation*}
x\in U\text{and}\,(x,x^*)\,\text{is monotonically related to $\gra A\cap U\times X^*$}
\Rightarrow (x,x^*)\in\gra A
\end{equation*}
holds.  We do not know if every maximally monotone operator
 is necessarily of type (FPV) in a general Banach space
  \cite{Si2, BorVan}, but it is true in reflexive spaces.
Every operator of type (NI)  is of type (FPV) (see
\cite[Theorem~9.10(b)]{Si5}).
\begin{fact}[Simons]
\emph{(See \cite[Theorem~44.2]{Si2}.)}
\label{f:referee0pv}
Let $A:X\To X^*$ be a maximally monotone
operator that is of type (FPV). Then
\begin{align*}
\overline{\dom A}=\overline{\conv \big(\dom A\big)}=\overline{P_X\big(\dom F_A\big)}.
\end{align*}
\end{fact}

Now we introduce some properties of operators of type (NI) (i.e., type (D)).
\begin{fact}[Gossez]\label{GozztyD:1}
{\rm (See \cite[Theorem~3.1]{Gossez3} or \cite[Theorem~48.4(b)]{Si2}.)}
Let $f:X\rightarrow \RX$ be proper lower semicontinuous and convex.
Then $\partial f$ is maximally monotone of type (NI).
\end{fact}

The following result provides a sufficient condition for
 the sum operator to be of type (NI).
\begin{fact}
\emph{(See \cite[Theorem~1.2]{MarSva2} and \cite[Corollary~3.5]{ZalVoiRe} or \cite[Corollary~18]{SiQ8}.)}
\label{NISum}
Let $A, B:X\rightrightarrows X^*$ be
maximally monotone of type (NI).  Assume that $\bigcup_{\lambda>0}
\lambda\left[\dom A-\dom B\right]$ is a closed subspace.
Then $A+B$ is maximally monotone of type (NI).
\end{fact}

\begin{fact}[Marques Alves and Svaiter]
\emph{(See \cite[Theorem~4.4]{MarSva}.)}
\label{MSRDC}
Let $A:X\rightrightarrows X$ be
maximally monotone of type (NI), and $F:X\times X^*\rightarrow\RX$
 be proper (norm) lower semicontinuous and convex.
  Let $B:X^{*}\rightrightarrows X^{**}$ be defined by
\begin{align*}
\gra B:=\big\{(x^*, x^{**})\in X^*\times X^{**}\mid
\langle x^*-a^*, x^{**}-a\rangle\geq0,\,\forall (a,a^*)\in\gra A\big\}.
\end{align*} Assume that
$F$ is a representer for $\gra A$.
Then $F^*$ is a representer for $\gra B$ .
\end{fact}

\begin{fact}[Phelps and Simons]
\emph{(See \cite[Corollary 2.6 and Proposition~3.2(h)]{PheSim}.)}\label{FPS:1}
Let $A\colon X\rightarrow X^*$ be  monotone and linear. Then $A$ is
maximally monotone and continuous.
\end{fact}

Combining the above Fact~\ref{FPS:1},
we have the following result.
\begin{fact}[Bauschke and Simons]{\rm (See
\cite[Theorem~2.3 and Corollary~3.7]{BS1}.)}\label{PropLe:2}
Let $A:X\rightrightarrows X^*$ be a monotone
linear relation with $\ran A=X^*$.
 Assume that $A$ is at most single-valued.
 Then $A^{-1}$ is single-valued and linear
 monotone  $X^*$ to $X^{**}$ with $\dom A^{-1}=X^*$.
  In consequence, $A^{-1}$ is continuous and then $A$ is ultramaximally monotone.

\end{fact}

During the 1970s Brezis
and Browder presented a now classical characterization of maximal
monotonicity of monotone linear relations in reflexive spaces
 \cite{BB76,Brezis-Browder,Si3}.
The following result is their result in the setting of a general real Banach space.
(See also \cite{Si3} and \cite{Si7} for Simons' recent extensions in
(SSDB)
spaces as defined in \cite[\S21]{Si2} and of Banach SNL spaces.)
 Recently,  Stephen Simons strengthens Fact~\ref{thm:bbwy} in \cite{Si8}.
\begin{fact}[Brezis-Browder in general Banach space] \label{thm:bbwy}
\emph{(See \cite[Theorem~4.1]{BBWY5}.)}
Let $A\colon X\rightrightarrows X^*$ be a monotone linear
 relation such that $\gra A$ is closed.
Then  $A$ is maximally monotone of type (NI) if
and only if  $A^*$ is monotone.
\end{fact}

The next lemma is trivial but helpful.
\begin{lemma}\label{PropEXT:1}
Let $f:X\rightarrow\RX$ be proper lower semicontinuous and
 convex with $\inte\dom f\neq\varnothing$. Assume that
$\partial f$ is ultramaximally monotone.
Then $X$ is reflexive.
\end{lemma}

\begin{proof} By the assumption, let $x_0\in \inte\dom f$.
 By \cite[Proposition~3.3]{ph}, there exist
$\delta, M>0$ such that
\begin{align}
\sup_{y\in x_0+\delta B_X} f(y)\leq M.\label{SDRL:c1}
\end{align}
Now we show that
\begin{align}
x_0+\delta B_{X^{**}}\subseteq\dom f^{**}.\label{SDRL:c2}
\end{align}
Let $x^{**}\in x_0+\delta B_{X^{**}}$.
By Goldstine's theorem (see \cite[Theorem~2.6.26]{Megg}), there exists
a net
 $(x_{\alpha})_{\alpha\in I}$ in $x_0+\delta B_X$
such that $x_{\alpha}\weakstarly x^{**}$.  Then by \eqref{SDRL:c1}, we have
\begin{align*}
f^{**}(x^{**})\leq\liminf f^{**}(x_{\alpha})=\liminf f(x_{\alpha})\leq M.
\end{align*}
Hence $x^{**}\in\dom f^{**}$ and thus \eqref{SDRL:c2} holds.

By Br{\o}ndsted-Rockafellar theorem (see \cite[Theorem~3.1.2]{Zalinescu}),
\eqref{SDRL:c2} shows that
\begin{align}
x_0+\delta B_{X^{**}}\subseteq\dom f^{**}\subseteq
\overline{\ran\partial f^*}.\label{CorLemmC:1}
\end{align}
On the other hand, \cite[Theorem~2.4.4(iv)]{Zalinescu}
implies that $\gra (\partial f)^{-1}
\subseteq\gra f^*$, by the assumption that
$\partial f$ is ultramaximally monotone, we have
\begin{align*}\gra \big(\partial f\big)^{-1}
=\gra \partial f^*.
\end{align*}
Thus, \eqref{CorLemmC:1} shows that $x_0+\delta B_{X^{**}}
\subseteq\overline{\dom \partial f}\subseteq X$.
Then we have $X^{**}=X$.  Hence we have $X$ is reflexive.
\end{proof}

\section{Our main results}\label{s:main}

\subsection{Properties of ultramaximally monotone operators}
In Theorem~\ref{PropLeP:1}, we provide a sufficient condition for
 the sum operator to be ultramaximally monotone.
 We first need the following two technical results.

\begin{lemma}\label{Co:1}
Let $A, B\colon X\To X^*$ be  maximally monotone,
 and suppose that
$\bigcup_{\lambda>0}\lambda\left[\dom A-\dom B\right]$
 is a closed subspace of $X$.
Set
\begin{align*}
E&:=\big\{x\in X\mid \exists x^*\in X^*\,
\text{such that}\,\, F^*_A(x^*,x)<+\infty\big\}
\end{align*}
and
\begin{align*}
F&:=\big\{x\in X\mid \exists x^*\in X^*\,
\text{such that}\,\, F^*_B(x^*,x)<+\infty\big\}.
\end{align*}
Then
\begin{equation*}\bigcup_{\lambda>0}\lambda\left[\dom A-\dom B\right]=
\bigcup_{\lambda>0}\lambda\left[E-F\right].
\end{equation*}
Moreover, if $A$ and $B$ are of  type (FPV), then we have
\begin{align*}\bigcup_{\lambda>0}\lambda\left[\dom A-\dom B\right]=
\bigcup_{\lambda>0}\lambda\left[P_{X}\dom F_A-P_{X}\dom F_B\right].\end{align*}
\end{lemma}
\allowdisplaybreaks
\begin{proof}
Using Fact~\ref{fa:Fitza2}, we see that
\begin{align*}
&\bigcup_{\lambda>0} \lambda\left[\dom A-\dom B\right]\subseteq
\bigcup_{\lambda>0} \lambda\left[E-F\right]
\subseteq\bigcup_{\lambda>0}
\lambda\left[\overline{\conv(\dom A)}-\overline{\conv(\dom B)}\right]\\
&\subseteq\bigcup_{\lambda>0}
\lambda\left[\overline{\conv(\dom A)-\conv(\dom B)}\right]
=
\bigcup_{\lambda>0}\lambda \overline{\left[\conv(\dom A-\dom B)\right]}\\
&\subseteq
\overline{ \bigcup_{\lambda>0}\lambda\left[\conv(\dom A-\dom B)\right]}
=\bigcup_{\lambda>0} \lambda
\left[\dom A-\dom B\right]\quad \text{(using the assumption)}.
\end{align*}
Hence
$\bigcup_{\lambda>0}\lambda\left[\dom A-\dom B\right]=
\bigcup_{\lambda>0}\lambda\left[E-F\right].$

Now assume that $A,B$ are of type (FPV).
Then by Fact~\ref{f:referee0pv}, we have
\begin{align*}
&\bigcup_{\lambda>0} \lambda\left[\dom A-\dom B\right]\subseteq
\bigcup_{\lambda>0} \lambda\left[P_{X}\dom F_A-P_{X}\dom F_B\right]
\subseteq\bigcup_{\lambda>0} \lambda\left[\overline{\dom A}-\overline{\dom B}\right]\\
&\subseteq\bigcup_{\lambda>0} \lambda\left[\overline{\dom A-\dom B}\right]\subseteq
\overline{\bigcup_{\lambda>0} \lambda\left[\dom A-\dom B\right]}\\
&=\bigcup_{\lambda>0} \lambda\left[\dom A-\dom B\right]\quad \text{(using the  assumption)}.
\end{align*}
\end{proof}

 Proposition~\ref{F12} below was first
established by Bauschke, Wang and Yao in
\cite[Proposition~5.9]{BWY3} when $X$ is  a reflexive space.
We now provide a nonreflexive version.

\begin{proposition}\label{F12}
Let $A,B\colon X\To X^*$  be maximally monotone and
suppose that
$\bigcup_{\lambda>0} \lambda\left[\dom A-\dom B\right]$ is a closed subspace of $X$.
Then
$ F_A\Box_2F_B$ is proper,
norm$\times$weak$^*$ lower semicontinuous and convex,
and the partial infimal convolution is exact everywhere.
\end{proposition}
\begin{proof}
Define $F_1,F_2\colon X\times X^* \to \RX$ by
\begin{align*}
F_1:(x,x^*)\mapsto F_{A}^{*}(x^*,x),\quad
F_2:(x,x^*)\mapsto F_{B}^{*}(x^*,x).\end{align*}
Since $F_{A}$ and $F_{B}$ are norm-weak$^*$ lower semicontinuous,
\begin{align}F_1^*(x^*,x)=F_{A}(x,x^*),\quad F_2^*(x^*,x)=F_{B}(x,x^*),
\quad \forall(x,x^*)\in X\times X^*.\label{sn:2}\end{align}
Take $(x,x^*)\in X\times X^*$.
By Fact~\ref{fa:Fitza2} and Fact~\ref{f:Fitz},
$$ \big(F_1\Box_2F_2\big)(x,x^*)
\geq\langle x,x^*\rangle>-\infty.$$
In view of Lemma~\ref{Co:1},
\begin{align*}\bigcup_{\lambda>0}
\lambda\left[P_{X}\dom F_1-P_{X}\dom F_2\right]
=\bigcup_{\lambda>0}
\lambda\left[\dom A-\dom B\right]\;\; \text{is a closed subspace}.\end{align*}
By  Fact~\ref{F4} and \eqref{sn:2},
\begin{align*}
&\big(F_1\Box_2F_2\big)^*(x^*,x)
=\min_{y^*\in X^*}\left[F_1^*(x^*-y^*,x)+F_2^*(y^*,x)\right]\\
&=\min_{y^*\in X^*}
\left[F_A(x,x^*-y^*)+F_B(x,y^*)\right]
=\big(F_A\Box_2F_B\big) (x,x^*).
\end{align*}
Hence $ F_A\Box_2F_B$ are proper,
norm$\times$weak$^*$ lower semicontinuous and convex,
and the partial infimal convolution is exact.
\end{proof}

Now we come to our first main result.
\begin{theorem}\label{PropLeP:1}
Let $A, B:X\rightrightarrows X^*$ be
maximally monotone.  Assume that $\bigcup_{\lambda>0}
\lambda\left[\dom A-\dom B\right]$ is a closed subspace.
Suppose that $A$ is ultramaximally monotone, and
 that $B$ is of type (NI).  Then $A+B$ is  ultramaximally monotone.

\end{theorem}

\begin{proof}Clearly, $A+B$ is monotone and $A$ is of type (NI).
Fact~\ref{NISum} shows that $A+B$ is maximally monotone of type (NI).
Define $K:X\times X^*\rightarrow\RX$ by $K:=F_A\Box_2 F_B$.
Proposition~\ref{F12} and Fact~\ref{f:Fitz}
  imply that $K$ is a representative  of $\gra (A+B)$.

Let $(x^*, x^{**})\in X^*\times X^{**}$
be such that $K^*(x^*, x^{**})=\langle x^*, x^{**}\rangle$.
\cite[Theorem~9.10(b)]{Si5} implies that $A$ and $B$ are of type (FPV).  Then
by Lemma~\ref{Co:1}, Fact~\ref{f:Fitz} and Fact~\ref{F4}, there exists $y^*\in X^*$ such that
\begin{align*}
K^*(x^*, x^{**})=F^*_A(y^*, x^{**})+F^*_B(x^*-y^*, x^{**}).
\end{align*}
Then \cite[Theorem~1.2]{MarSva2} shows that \begin{align}
F^*_A(y^*, x^{**})=\langle y^*, x^{**}\rangle\quad
\text{and}\quad F^*_B(x^*-y^*, x^{**})
=\langle x^*-y^*, x^{**}\rangle.\label{lrsee:12}
\end{align}
Since $A$ is ultramaximally monotone,
by  Fact~\ref{MSRDC} and \eqref{lrsee:12},  we have
\begin{align}
x^{**}\in X\quad\text{and}\quad y^*\in Ax^{**}.\label{lrsee:13}
\end{align}
Then combining \eqref{lrsee:12}, Fact~\ref{f:Fitz} and Fact~\ref{fa:Fitza2},
$x^*-y^*\in B x^{**}$ and hence $(x^{**}, x^*)\in\gra (A+B)$.
Hence
\begin{align}
\big\{(x^{**}, x^*)\in X^{**}
\times X^*|K^*(x^*, x^{**})=\langle x^*, x^{**}\rangle\big\}\subseteq\gra (A+B).\label{lrsee:14}
\end{align}
Since $K$ is a representative  of $\gra (A+B)$ and $A+B$ is of type (NI),
Fact~\ref{MSRDC} shows that $A+B$ is
 ultramaximally monotone.
\end{proof}

Let $A:X\rightrightarrows X^*$ be monotone.
For convenience,  we  defined $\Phi_A$ on $X^{**}\times X^*$
\label{symbol:41}
by
\begin{equation*}
\Phi_A\colon
(x^{**},x^*)\mapsto \sup_{(a,a^*)\in\gra A}
\big(\scal{x^{**}}{a^*}+\scal{a}{x^*}-\scal{a}{a^*}\big).
\end{equation*}
Then we have $\Phi_A|_{X\times X^*}=F_A$.

Now we present some characterizations of
the interior and the closure of the range of an ultramaximally monotone
operator, which generalizes Simons'
results in a reflexive space (see \cite[Theorem~31.2 and Lemma~31.1]{Si2}).

\begin{theorem}\label{domFdm:1}
Let $A:X\rightrightarrows X^*$ be
ultramaximally monotone.  Then
\begin{align}\inte\ran A&=\inte\left[\conv\ran A\right]=\inte \left[P_{X^*}
\dom F_A\right]=\inte\left[P_{X^*}
\dom \Phi_A\right]\label{domFdm:1a}\\
\overline{\ran A}&=\overline{\conv \ran A}
=\overline{P_{X^*}\dom F_A}=\overline{P_{X^*}\dom \Phi_A}.
\label{domFdm:1b}\end{align}
\end{theorem}

\begin{proof}We first show \eqref{domFdm:1a}.
Fact~\ref{f:Fitz} implies that
\begin{align}
\ran A\subseteq \conv\left[\ran A\right]
\subseteq P_{X^*}\left[\dom F_A\right]
\subseteq P_{X^*}\left[\dom \Phi_A\right].\label{domFdm:E1}
\end{align}
Define $B:X^{*}\rightrightarrows X^{**}$
by $\gra B:=\gra A^{-1}$.  By the assumption,
$B$ is maximally monotone.
By Fact~\ref{interodom}, we have
\begin{align}
\inte\dom B=\inte\left[\conv\dom B\right]
=\inte\left[P_{X^*}\dom F_B\right].
\label{domFdm:E2}
\end{align}
By the definition of $B$, we have \begin{align}
\dom B=\ran A\quad\text{and}\quad\Phi_A (x^{**},
 x^*)=F_B(x^*, x^{**}),\quad\forall (x^{**}, x^*)\in X^{**}\times X^*.\label{domFdm:E3}
\end{align}
 Then \eqref{domFdm:E2} shows that
\begin{align*}
 \inte\ran A=\inte\left[\conv\ran A\right]
 =\inte P_{X^*}\left[\dom \Phi_A\right].
\end{align*}
Thus combining \eqref{domFdm:E1}, we have
$\inte\ran A=\inte\left[\conv\ran A\right]=\inte \left[P_{X^*}
\dom F_A\right]=\inte\left[P_{X^*}
\dom \Phi_A\right]$.

Now we show \eqref{domFdm:1b}.
Fact~\ref{CoHull} implies that
\begin{align*}
\overline{\conv \dom B}=\overline{P_{X^*}\dom F_B}.
\end{align*}
Thus \eqref{domFdm:E3} implies that
\begin{align*}
\overline{\conv \ran A}=\overline{P_{X^*}\dom \Phi_A}.
\end{align*}
Thus combining \eqref{domFdm:E1}, we have
\begin{align}\overline{\conv \ran A}
=\overline{P_{X^*}\dom F_A}=\overline{P_{X^*}\dom \Phi_A}.
\label{domFdm:E4}
\end{align}
Since $A$ is of type (NI),
\cite[Theorem~43.2]{Si2} and \eqref{domFdm:E4} show
that $\overline{\ran A}=\overline{\conv \ran A}
=\overline{P_{X^*}\dom F_A}=\overline{P_{X^*}\dom \Phi_A}$
 and hence \eqref{domFdm:1b} holds.
\end{proof}

\begin{remark}
We cannot significantly weaken the conditions in Theorem~\ref{domFdm:1}.
For instance, we cannot replace
``ultramaximally monotone" by ``of type (NI)".
In a nonreflexive space
 there always exists a continuous,
  coercive,  and convex function $f$ such that
  $\inte\left[\ran\partial f\right]$ is not convex
 (where $\partial f$ is of type (NI) by Fact~\ref{GozztyD:1}, see also \cite[Theorem~3.1]{BFV1}
 and \cite[page~169]{Si2} for more information),
 but \eqref{domFdm:1a} grantees that
 $\inte\left[\ran\partial f\right]$ is convex.
\end{remark}

The following result is very useful, which allows us
 to show that every ultramaximally
monotone operator is of type (NA) (see Theorem~\ref{TyNA:2} below).
\begin{corollary}\label{GozztyD:2}
Let $A:X\rightrightarrows X^*$ be
ultramaximally monotone.   Then $A+J$ is
 ultramaximally monotone and $\ran (A+J)=X^*$.

\end{corollary}
\begin{proof}
By Fact~\ref{GozztyD:1}, $J$ is of type (NI).
Then  applying Theorem~\ref{PropLeP:1}, we have $A+J$
is ultramaximally monotone.
Now we show that $\ran (A+J)=X^*$.

By \cite[Eq.(23.9), page~101]{Si2},
$\dom F_J=X\times X^*$.  Then $X^*=P_{X^*}
\left[\dom(F_A\Box_2 F_J)\right]$.
Thus, \cite[Lemma~23.9]{Si2} implies that
$X^*=P_{X^*}\left[\dom F_{A+J}\right]$.
Since $A+J$ is ultramaximally monotone, Theorem~\ref{domFdm:1}
implies that \begin{align*}\inte\ran (A+J)
=\inte\left[P_{X^*}\dom F_{A+J}\right]=\inte X^*=X^*.
\end{align*}
Hence $\ran (A+J)=X^*$.
\end{proof}

\begin{corollary}\label{Coeriv:1}
Let $A:X\rightrightarrows X^*$ be
 ultramaximally monotone with $0\in\dom A$. Assume that
$\lim_{\|x\|\rightarrow\infty}\inf
\frac{\langle x, Ax\rangle}{\|x\|}=+\infty$.
Then $\ran A=X^*$.
\end{corollary}
\begin{proof}
We first show that
\begin{align}\label{URan:C1}
\{0\}\times X^*\subseteq\dom F_A.
\end{align}
Since $0\in\dom A$, there exists
$x^*_0\in X^*$ such that $(0,x^*_0)\in\gra A$.
Let $x^*\in X^*$.  By the assumption, there exists
$\rho>0$ such that
\begin{align}
\langle a, a^*\rangle\geq \|x^*\|\cdot\|a\|,
\quad\forall \|a\|\geq\rho,\,
(a,a^*)\in\gra A.\label{URan:C2}
\end{align}
Let $(a,a^*)\in\gra A$.

\emph{Case 1:} $\|a\|<\rho$.

By $(0,x^*_0)\in\gra A$, we have $\langle a,a^*-x^*_0
\rangle=\langle a-0,a^*-x^*_0\rangle\geq0$.
We have \begin{align*}\langle a, x^*\rangle-
\langle a,a^*\rangle
\leq\langle a, x^*\rangle-\langle a,x_0^*
\rangle\leq \rho\|x^*-x^*_0\|.
\end{align*}

\emph{Case 2:} $\|a\|\geq \rho$.

Using \eqref{URan:C2}, we have
\begin{align*}\langle a, x^*\rangle-\langle a,a^*\rangle
= \langle a, x^*\rangle-\|x^*\|\cdot\|a\|\leq 0.
\end{align*}
Thus combining the above two cases,
\begin{align*}
F_A(0, x^*)=\sup_{(a,a^*)\in\gra A}\{\langle a, x^*\rangle
-\langle a,a ^*\rangle\}\leq \rho\|x^*-x^*_0\| .
\end{align*}
Thus $(0,x^*)\in\dom F_A$.
Hence \eqref{URan:C1} holds and then
$P_{X^*}
\dom F_A= X^*$.
Thus Theorem~\ref{domFdm:1} implies that
\begin{align*}
\inte \ran A=\inte \left[P_{X^*}
\dom F_A\right]=\inte X^*=X^*.
\end{align*}
Thus $\ran A=X^*$.
\end{proof}

Corollary~\ref{Coeriv:1} was first proved
by Browder in a reflexive space (see \cite[Theorem~3]{Browder1}).

\begin{corollary}
Let $A:X\rightrightarrows X^*$ be  monotone. Let $\delta>0$ and $\alpha>1$.
Assume that \begin{align*}\langle x-y, x^*-y^*\rangle\geq\delta\|x-y\|^{\alpha},\quad
\forall (x,x^*), (y,y^*)\in\gra A.
\end{align*}
Then $A$ is ultramaximally monotone if and only if
$\ran A=X^*$.
\end{corollary}

\begin{proof}
``$\Rightarrow$":
Let $(x_0,x^*_0)\in\gra A$. Define $B:X\rightrightarrows
X^*$ by $\gra B:=\gra A-\{(x_0, x^*_0)\}$.
Then by the assumption, we have $B$ is
 ultramaximally monotone with $(0,0)\in\gra B$ and
\begin{align}\langle b, b^*\rangle\geq\delta\|b\|^{\alpha},\quad
\forall (b,b^*)\in\gra B.\label{URan:d1}
\end{align}
Thus,
$\lim_{\|x\|\rightarrow\infty}\inf
\frac{\langle x, Bx\rangle}{\|x\|}=+\infty$.
Then by Corollary~\ref{Coeriv:1},
$\ran B=X^*$ and hence $\ran A= X^*$ by the definition of $B$.

``$\Leftarrow$":
We first show that $A^{-1}$ is single-valued on $X^*$.
Let $x,y\in A^{-1}x^*$.  Then
$(x,x^*), (y, x^*)\in\gra A$.  Then by the assumption,
$0=\langle x-y, x^*-x^*\rangle\geq\delta\|x-y\|^{\alpha}$.
   Thus, $x=y$ and hence $A^{-1}$ is single-valued on
   $X^*$ since $\ran A=X^*$.  By the assumption again,
we have
\begin{align*}
 \|A^{-1}x^*-A^{-1}y^*\|\cdot\|x^*-y^*\|
 \geq\langle A^{-1}x^*-A^{-1}y^*, x^*-y^*
 \rangle\geq\delta\| A^{-1}x^*-A^{-1}y^*\|^{\alpha},\quad
\forall x^*, y^*\in X^*.
\end{align*}
Thus
\begin{align*}
\|x^*-y^*\|\geq\delta\| A^{-1}x^*-A^{-1}y^*\|^{\alpha-1},\quad
\forall x^*, y^*\in X^*.
\end{align*}
Since $\alpha>1$, $A^{-1}$ is continuous and
then $A^{-1}$ is maximally monotone from $X^*$ to $X^{**}$.
Hence $A$ is ultramaximally monotone.
\end{proof}

Let $A:X\rightrightarrows X^*$ be such that
$\dom A\neq\varnothing$.  We say that $A$ is \emph{rectangular}
if $\dom A\times\ran A\subseteq\dom F_A$
(see \cite{BreHara,  Si2, Zeidler2B, BC2011,BWY4}
 for more information on rectangular operators).

The proof of Theorem~\ref{RectaE:1} closely follows the lines of that of
\cite[Corollary~31.6]{Si2}.

\begin{theorem}[The Brezis-Haruax condition
in general Banach space]\label{RectaE:1}
Let $A, B:X\rightrightarrows X^*$ be monotone
with $\dom A\cap\dom B\neq\varnothing$.  Assume that $A+B$ is
ultramaximally monotone.  Suppose that
one of the following conditions holds:
\begin{enumerate}
\item  \label{RectaE:1a}$A$ and $B$ are rectangular.
\item \label{RectaE:1b} $\dom A\subseteq\dom B$ and $B$ is rectangular.
\end{enumerate}
Then \begin{align*}\inte \ran (A+B)=
\inte\left[\ran A+\ran B\right]\quad\text{and}\quad
\overline{\ran A+\ran B}=\overline{\ran (A+B)}.
\end{align*}
\end{theorem}

\begin{proof}
We first show that
\begin{align}
\ran A+\ran B\subseteq P_{X^*}\dom F_{A+B}.\label{RectaE:1e1}
\end{align}
Let $x^*\in \ran A+\ran B$, and $x_0\in\dom A\cap\dom B$.
Thus there exists $x_1^*\in \ran A, x_2^*\in\ran B$ such that $x^*=x^*_1
+x^*_2$.  Then we have $(x_0, x^*_1)\in\dom A\times\ran A$ and $(x_0, x^*_2)
\in\dom B\times\ran B$.

Now we consider two cases.

\emph{Case 1}: \ref{RectaE:1a} holds.

We have $(x_0, x^*_1)\in\dom F_A$ and $(x_0, x^*_2)
\in \dom F_B$.  Since $F_A\Box_2 F_B(x_0, x^*)\leq F_A(x_0, x^*_1)
+F_B(x_0, x^*_2)<+\infty$, $(x_0, x^*)\in\dom F_A\Box_2 F_B$.
Then \cite[Lemma~23.9]{Si2} implies that $(x_0, x^*)\in\dom F_{A+B}$
and thus $x^*\in P_{X^*}\dom F_{A+B}$.  Hence \eqref{RectaE:1e1} holds.

\emph{Case 2}: \ref{RectaE:1b} holds.

Since $x^*_1\in\ran A$, there exists $x_1\in X$ such that $
(x_1, x^*_1)\in\gra A$.
By the assumption, $(x_1,x^*_2)\in\dom B\times\ran B\subseteq\dom F_B$.
Similar to the corresponding lines in Case 1,
we have  $x^*\in P_{X^*}\dom F_{A+B}$ and thus \eqref{RectaE:1e1} holds.

Combining all the above cases, \eqref{RectaE:1e1} holds.

By Theorem~\ref{domFdm:1} and \eqref{RectaE:1e1}
\begin{align*}
\inte\left[\ran A+\ran B\right]\subseteq
\inte\left[P_{X^*}\dom F_{A+B}\right]=\inte\ran (A+B)
\subseteq \inte\left[\ran A+\ran B\right].
\end{align*}
Hence $\inte\left[\ran A+\ran B\right]=\inte\ran (A+B)$.

By Theorem~\ref{domFdm:1} and \eqref{RectaE:1e1} again,
\begin{align*}
\ran A+\ran B\subseteq P_{X^*}\dom F_{A+B}
\subseteq\overline{\ran (A+B)}\subseteq\overline{
\ran A+\ran B}.
\end{align*}
Hence $\overline{\ran A+\ran B}=\overline{\ran (A+B)}$.
\end{proof}

\begin{remark}
 Brezis and Haraux proved Theorem~\ref{RectaE:1}
  in the setting of a Hilbert space (see
\cite[Theorems~3\&4, pp.~173--174]{BreHara}).
Reich extent the above result to a reflexive space (see \cite[Theorem~2.2]{Reich}).

\end{remark}
Let $A:X\rightrightarrows X^*$ be monotone.
We say $A$ is \emph{of type (NA)} (where (NA) stands for
``negative alignment") \cite{BS1} if
for every $(x,x^*)\in X\times X^*\backslash\gra A$,
there exists $(a,a^*)\in\gra A$ such that $a\neq x, a^*\neq x^*$ and
\begin{align*}
\langle x-a, x^*-a^*\rangle=-\|x-a\|\cdot\|x^*-a^*\|.
\end{align*}

\begin{proposition}\label{TyNA:1}
Let $A:X\rightrightarrows X^*$ be ultramaximally
monotone and $(x,x^*)\in X\times X^*$.
Then there exists $(a,a^*)\in\gra A$ such that
$\|x^*-a^*\|=\|x-a\|$ and
\begin{align*}
\|x-a\|^2+\|x^*-a^*\|^2+2\langle x-a, x^*-a^*\rangle=0.
\end{align*}

\end{proposition}
\begin{proof}
Define $B:X\rightrightarrows X^*$ by $B:=A(\cdot +x)$.
By the assumption, we have $B$ is ultramaximally monotone.
Then Corollary~\ref{GozztyD:2} implies that $x^*\in\ran (B+J)$.
 Thus there exists $b\in X$ such that $x^*\in Bb+Jb= A(b+x)+Jb$.
Let $a:=b+x$. Thus $b:=a-x$.  We have $x^*\in Aa+J(a-x)$.
Let $a^*\in Aa$ such that $x^*\in a^*
+J(a-x)$. Hence we have
$x^*-a^*\in
J(a-x)$ and then
\begin{align*}
\langle x^*-a^*, a-x\rangle=\|x^*-a^*\|\cdot\|a-x\|\quad
\text{and}\quad \|x^*-a^*\|=\|a-x\|.
\end{align*}
Then
\begin{align*}
\|x-a\|^2+\|x^*-a^*\|^2+2\langle x-a, x^*-a^*\rangle
=2\|x-a\|^2-2\|x^*-a^*\|\cdot\|a-x\|=0.
\end{align*}
\end{proof}

Now we show that every ultramaximally monotone operator is of type (NA).
\begin{theorem}\label{TyNA:2}
Let $A:X\rightrightarrows X^*$ be ultramaximally monotone. Then $A$ is of type (NA).

\end{theorem}
\begin{proof}
Let $(x,x^*)\in X\times X^*\backslash\gra A$.
 Proposition~\ref{TyNA:1} implies that
there exists $(a,a^*)\in\gra A$ such that $\|x^*-a^*\|=\|x-a\|$ and
$2\|x-a\|\cdot\|x^*-a^*\|+2\langle x-a, x^*-a^*\rangle=0
$. Thus
\begin{align*}\langle x-a, x^*-a^*\rangle=-\|x-a\|\cdot\|x^*-a^*\|.
\end{align*}
Since $(x,x^*)\not\in\gra A$ and $\|x^*-a^*\|=\|x-a\|$,
we have $\|x^*-a^*\|=\|x-a\|\neq0$.
Hence $x\neq a$ and $x^*\neq a^*$.
\end{proof}

\begin{remark}
In \cite[Remark~29.4]{Si2}, Simons shows that not every operator of type (NI)
 is necessarily of type (NA) (see the operator
$A:=\partial \iota_X$ when $X$ is nonreflexive).
Hence we cannot significantly weaken the conditions in Theorem~\ref{TyNA:2}.
\end{remark}

\begin{corollary}[Bauschke and Simons] \emph{(See \cite[Theorem~3.5]{BS1}.)}
\label{TyNA:3}
Let $A:X\rightrightarrows X^*$ be an ultramaximally monotone  linear relation.
 Assume that $A$ is at most single-valued. Then $A$ is of type (NA).

\end{corollary}

\subsection{Linear continuous operator}
In this subsection, we present some sufficient conditions for
 a Banach space to be reflexive  by a linear continuous and ultramaximally monotone operator.

Let $A:X\rightarrow X^*$ be a linear continuous operator.
Define $P, S:X\rightarrow X^*$ by
\begin{align}\label{lrsee:c7}
P:=\frac{A+A^*}{2}\quad\text{and}\quad S:=\frac{A-A^*}{2},
\end{align}
 respectively.  When $A$ is  monotone,
we apply use the well known fact (see, e.g., \cite{PheSim}) to obtain that
\begin{equation}
\nabla f= P,\label{lrsee:c8}
\end{equation} where $f:=\frac{1}{2}\langle x, Ax\rangle,\quad\forall x\in X$.

\begin{lemma}\label{ProLem:2}
Let $A:X\rightarrow X^*$ be linear continuous and ultramaximally monotone.
Let $S$ be defined as in \eqref{lrsee:c7}.
 Then $S$ is  ultramaximally monotone.
 \end{lemma}
\begin{proof}
Suppose to the contrary that $S$ is not
 ultramaximally monotone.
Since $S$ is maximally monotone, there exists $(x^{**}, x^*)\in X^{**}
\times X^*$ such that $x^{**}\notin X$ and  $(x^{**}, x^*)$
is monotonically related to $\gra S$.

We have $\gra A:=\gra(P+S)\subseteq\gra (P^*+S)$.
Fact~\ref{GozztyD:1} and \eqref{lrsee:c8}
show that $P$ is maximally monotone of type (NI).
Fact~\ref{thm:bbwy} shows that $P^*$ is monotone.
Then we have $(x^{**}, P^*x^{**}+x^*)$ is  monotonically related to $\gra A$. By
the assumption, $(x^{**}, P^*x^{**}+x^*)\in\gra A$,
which contradicts that $x^{**}\notin X$.
\end{proof}

\begin{corollary}\label{ProLem:L2}
Let $A:X\rightarrow X^*$ be linear continuous
and monotone. Assume that $\ran A=X^*$.
Let $S$ be defined as in \eqref{lrsee:c7}.
 Then $S$ is  ultramaximally monotone.

\end{corollary}
\begin{proof}
Combine Fact~\ref{PropLe:2} and Lemma~\ref{ProLem:2} directly.
\end{proof}

\begin{proposition}\label{PropLe:P4}
Let $A:X\rightarrow X^*$ be linear continuous and
 ultramaximally monotone, and $S$ be defined as
 in \eqref{lrsee:c7}.
Suppose that  there exists a proper lower
semicontinuous and convex function $g:X\rightarrow\RX$ such that
$\inte\dom g\neq\varnothing$ and
$\partial g- S$ is of type (NI). Then $X$ is reflexive.

\end{proposition}
\begin{proof}
Theorem~\ref{PropLeP:1} and Lemma~\ref{ProLem:2}
 imply that $\partial g= S+(\partial g-S)$ is  ultramaximally monotone.
Thus combining with Lemma~\ref{PropEXT:1},
$X$ is reflexive.
\end{proof}

\begin{corollary}\label{PropLe:5}
Let $A:X\rightarrow X^*$ be linear continuous and ultramaximally monotone.
Suppose that  there exists $\delta>0$
such that $\langle Ax, x\rangle\geq\delta\|x\|^2,
\quad\forall x\in X$. Then $X$ is reflexive.

\end{corollary}

\begin{proof} Let $P, S$ be defined as in \eqref{lrsee:c7}.
Now we show that
\begin{align}
\langle P^*x^{**}, x^{**}\rangle\geq\delta\|x^{**}\|^2,
\quad\forall x^{**}\in X^{**}.\label{lrsee:20}
\end{align}
Let $x^{**}\in X^{**}$.
\cite[Corollary~12.3.9, page~146]{Bthesis} implies that
there exists a bounded net
$(x_{\alpha})_{\alpha\in I}$ such that
$x_{\alpha}\weakstarly x^{**}$ and
 $Px_{\alpha}\longrightarrow P^*x^{**}$.
 Then we have
\begin{align*}
\delta\|x_{\alpha}\|^2\leq
\langle Ax_{\alpha}, x_{\alpha}\rangle
=\langle Px_{\alpha}, x_{\alpha}\rangle
\longrightarrow\langle P^*x^{**}, x^{**}\rangle.
\end{align*}
Then \cite[Theorem~2.6.14]{Megg}
implies that $\langle P^*x^{**}, x^{**}\rangle\geq
\delta\|x^{**}\|^2$.
Hence \eqref{lrsee:20} holds.

By \eqref{lrsee:20}, there exists
$\gamma>0$ such that $\gamma P^*-S^*$ is monotone.
Fact~\ref{thm:bbwy} shows that $\gamma P-S$
is maximally monotone of type (NI).
Let $f$ be defined as in \eqref{lrsee:c8}.
Then  $\gamma P-S=\partial \gamma f -S$.
Thus Proposition~\ref{PropLe:P4} implies that $X$ is reflexive.
\end{proof}

\begin{corollary}\label{PropLe:L5}
Let $A:X\rightarrow X^*$ be linear
continuous and monotone. Assume that $\ran A=X^*$.
Suppose that  there exists
 $\delta>0$ such that $\langle Ax, x\rangle\geq\delta\|x\|^2,
 \quad\forall x\in X$. Then $X$ is reflexive.

\end{corollary}

\begin{proof}
Combine Corollary~\ref{PropLe:5} and Fact~\ref{PropLe:2} directly.
\end{proof}

At last, we pose the following three interesting problems.
\begin{problem}\label{Prob:unt1}
Let $A: X\rightrightarrows X^*$ be maximally monotone
with $\inte\dom A\neq\varnothing$ and $\ran A=X^*$.
Is $A$ necessarily ultramaximally monotone?
\end{problem}

\begin{problem}\label{Prob:unta1}
If there exists a linear continuous and ultramaximally
monotone operator  defined on $X$,
is $X$ necessarily  reflexive?
\end{problem}

A  general problem is that:
\begin{problem}\label{Prob:unt2} If there exists an
ultramaximally monotone operator
with nonempty interior domain defined on $X$,
is $X$ necessarily  reflexive?
\end{problem}

For the subdifferential operator,
we
 have an affirmative answer to Problem~\ref{Prob:unt2} by Lemma~\ref{PropEXT:1}.

Saint Raymond presents the following
 interesting result related to
 Lemma~\ref{PropEXT:1} and Problem~\ref{Prob:unt2}
   in  \cite[Corollary~2.5]{SaiRay}:
\begin{quote}
Let $f:X\rightarrow\RX$ be proper
 with $\inte\dom f\neq\varnothing$ and $\partial f=X^*$.
Then $X$ is reflexive.
\end{quote}
Orihuela and  Ruiz Gal\'{a}n proved the above result when $f$ is supercoercive
(see \cite[Theorem~7]{OriGala}).

We wonder if there exists a direct way
 to prove that $\partial f$ is ultramaximally
 monotone when $f$ is lower semicontinuous and
 convex under Saint Raymond's assumption that
 $\inte\dom f\neq\varnothing$ and $\partial f=X^*$.

\section*{Acknowledgment}
The author thanks Dr.~Tony Lau and Department of Mathematical and Statistical Sciences at University of Alberta for supporting his visit and
providing excellent working conditions, which completed this research.

\end{document}